\documentclass[12pt]{article}
\usepackage{amsmath}
\usepackage{amsfonts}
\usepackage{amsthm}
\usepackage{graphicx}
\graphicspath{{Figures/}} 
\usepackage{overpic}
\usepackage{amssymb} 
\usepackage{pictex}
\usepackage{rotating}  
\usepackage{color}
\usepackage{cite} 

\usepackage{enumitem} 

\usepackage{accents} 

\usepackage{yhmath} 
\usepackage{etex} 

\usepackage{comment}
\numberwithin{equation}{section}

\newtheorem{theorem}{Theorem}[section]
\newtheorem{proposition}[theorem]{Proposition}

\newtheorem{Definition}[theorem]{Definition}
\newenvironment{definition}{\begin{Definition}\rm}{\end{Definition}}
\newtheorem{Remark}[theorem]{Remark}
\newenvironment{remark}{\begin{Remark}\rm}{\end{Remark}}
\newtheorem{RHproblem}[theorem]{RH problem}

\newtheorem{Example}[theorem]{Example}

\usepackage{color}

\newcommand{\B}{\mathbb{B}}
\newcommand{\C}{\mathbb{C}}

\newcommand{\N}{\mathbb{N}}
\renewcommand{\P}{\mathbb{P}}
\newcommand{\Q}{\mathbb{Q}}
\newcommand{\R}{\mathbb{R}}

\newcommand{\OO}{\mathcal O}

\newcommand{\QQ}{\mathcal Q}

\renewcommand{\Re}{{\rm Re} \,}

\def \area{\mbox{{\rm area}}}

\def \deg{\mbox{{\rm deg}}}
\def\det{\mathop{\mathrm{det}}\nolimits}

\def\Poly{\mathop{\mathrm{Poly}}\nolimits}

\renewcommand{\bar}{\overline}
\renewcommand{\tilde}{\widetilde}

\usepackage[bookmarksopen, naturalnames]{hyperref}
\begin{document}
\title{$C-$transfinite diameter}
\author{N. Levenberg and F. Wielonsky}

\maketitle 

\begin{abstract} We give a general formula for the $C-$transfinite diameter $\delta_C(K)$ of a compact set $K\subset \C^2$ which is a product of univariate compacta where $C\subset (\R^+)^2$ is a convex body. Along the way we prove a Rumely type formula relating $\delta_C(K)$ and the $C-$Robin function $\rho_{V_{C,K}}$ of the $C-$extremal plurisubharmonic function $V_{C,K}$ for $C \subset (\R^+)^2$ a triangle $T_{a,b}$ with vertices $(0,0), (b,0), (0,a)$. Finally, we show how the definition of $\delta_C(K)$ can be extended to include many {\it nonconvex} bodies $C\subset \R^d$ for $d-$circled sets $K\subset \C^d$, and we prove an integral formula for $\delta_C(K)$ which we use to compute a formula for the $C-$transfinite diameter of the Euclidean unit ball $\B\subset \C^2$.
\end{abstract}

\section{Introduction} In the recently developed pluripotential theory associated to a convex body $C\subset (\R^+)^d$ (cf., \cite{BBL}), notions of $C-$extremal plurisubharmonic (psh) function $V_{C,K}$ and $C-$transfinite diameter $\delta_C(K)$ of a compact set $K\subset \C^d$ generalize the corresponding notions in the standard setting. Their definitions are recalled in the next section, and we include a brief discussion of Ma`u's recent work \cite{Sione} on $C-$transfinite diameter. We also recall the notion of $C-$Robin function $\rho_{V_{C,K}}$ associated to $V_{C,K}$ as defined in \cite{LM} for $C \subset (\R^+)^2$ a triangle $T_{a,b}$ with vertices $(0,0), (b,0), (0,a)$. The $C-$Robin function describes the precise asymptotic behavior of $V_{C,K}$; i.e., the behavior of $V_{C,K}(z)$ for $|z|$ large. 

In classical pluripotential theory, which corresponds to the special case where $C$ is the standard unit simplex $\Sigma\subset (\R^+)^d$, it is very difficult to find explicit formulas for extremal psh functions $V_K$ (and hence their Robin functions) or to find precise values of transfinite diameters $\delta_d(K)$ for $K\subset \C^d$. In 1962, Schiffer and Siciak \cite{SS} proved that if $K=E_1\times \cdots \times E_d\subset \C^d$ is a product of planar compact sets $E_j$, then $\delta_d(K) = \prod_{j=1}^d D(E_j)$ where $D(E_j)$ is the univariate transfinite diameter of $E_j$. Their proof used an intertwining of univariate Leja sequences for the sets $E_j$. Then in 1999, Bloom and Calvi \cite{BC} proved a more general result: if $K=E\times F$ where $E\subset \C^m$ and $F\subset \C^n$, then 
\begin{equation}\label{geneq} \delta_{n+m}(K)=\bigl( \delta_m(E)^m\cdot \delta_n(F)^n\bigr)^{\frac{1}{m+n}}.\end{equation}
Their proof used orthogonal polynomials associated to certain measures, called Bernstein-Markov measures, on $K$. In 2005, Calvi and Phung Van Manh \cite{CP} recovered the Bloom-Calvi result (\ref{geneq}) by generalizing the Schiffer-Siciak method in introducing ``block'' Leja sequences for the component sets. 

In \cite{Ru}, Rumely gave a remarkable formula relating transfinite diameter and Robin function in this classical setting. Using this formula, Blocki, Edigarian and Siciak \cite{BES} gave a very short proof of the general product formula (\ref{geneq}). In section 3, based on results in \cite{BBL} and \cite{LM}, we prove a Rumely type formula relating $\delta_C(K)$ and $\rho_{V_{C,K}}$ for $C=T_{a,b} \subset (\R^+)^2$ and we use this in section 4 to prove a formula for $\delta_C(K)$ when $K=E\times F$ is a product of univariate compacta. We modify the Bloom-Calvi proof using orthogonal polynomials in section 5 to give a product formula for the $C-$transfinite diameter when $C$ is a general convex body in $(\R^+)^2$. In particular, for such $C$ which are symmetric with respect to the line $y=x$, we obtain the striking result that the $C-$transfinite diameter of $K=E\times F$ is the same for these $C$. Finally, in section 6, we show how the $C-$transfinite diameter $\delta_C(K)$ can be extended to include many {\it nonconvex} bodies $C\subset \R^d$ for $d-$circled sets $K\subset \C^d$, and we exhibit an integral formula for $\delta_C(K)$. We use this to directly compute a formula for $\delta_{C_p}(\B)$ for the Euclidean unit ball $\B\subset \C^2$ for a natural one-parameter family of symmetric $C=C_p$ (section 6) which explicitly yields different values for different $p$. 

\section{$C-$transfinite diameter and $C-$Robin function} Let $C$ be a convex body in $(\R^+)^d$. We assume throughout that 
\begin{equation} \label{stdhyp} \epsilon \Sigma \subset C \subset \delta \Sigma \ \hbox{for some} \ \delta > \epsilon >0 \end{equation} where 
$$\Sigma:=\{(x_1,...,x_d)\in \R^d: 0\leq x_i \leq 1, \ \sum_{j=1}^d x_i \leq 1\}.$$ We set
\begin{align*}
\Poly(nC)=\{p(z) & =\sum_{J\in nC \cap \N^d}c_{J}z^{J}=\sum_{J\in nC \cap \N^d}c_{J}z_1^{j_1}\cdots z_d^{j_d},~c_{J}\in\C\},\quad n=1,2,\ldots
\end{align*}
and for a nonconstant polynomial $p$ we define 
$$ \deg_C(p)=\min\{ n\in\N\colon p\in \Poly(nC)\}.$$
Next, we define the logarithmic indicator function
$$H_C(z):=\sup_{J\in C} \log |z^J|:=\sup_{(j_1,...,j_d)\in C} \log\left(|z_1|^{j_1}\cdots |z_d|^{j_d}\right)$$
in order to define
$$L_C=L_C(\C^d):= \{u\in PSH(\C^d): u(z)- H_C(z) =O(1), \ |z| \to \infty \},$$ and 
$$L_C^+=L_C^+(\C^d)=\{u\in L_C(\C^d): u(z)\geq H_C(z) + C_u\}$$
where $PSH(\C^d)$ denotes the class of plurisubharmonic functions on $\C^d$. In particular, 
$$\hbox{if}  \ p\in \Poly(nC) \ \hbox{then}  \ u(z):=\frac{1}{\deg_C(p)}\log |p(z)|\in L_C.$$ These classes are generalizations of the classical Lelong classes $L:=L_{\Sigma}, \ L^+:=L_{\Sigma}^+$ when $C=\Sigma$. The $C$-extremal function of a compact set $K\subset\C^d$ is defined as the uppersemicontinuous (usc) regularization $V_{C,K}^*(z):=\limsup_{\zeta \to z} V_{C,K}(\zeta)$ of 
$$
V_{C,K}(z):=\sup\{u(z)\colon u\in L_C, u\leq 0 \hbox{ on } K\}.
$$
If $C=\Sigma$, we simply write $V_K:=V_{\Sigma,K}$. As in this classical setting, $V_{C,K}^*\equiv +\infty$ if and only if $K$ is pluripolar; and when this is not the case, the complex Monge-Amp\`ere measure $(dd^cV_{C,K}^*)^d$ is supported in $K$. We call $K$ {\it regular} if $V_K =V_K^*$; i.e., $V_K$ is continuous. This is equivalent to $V_{C,K}$ being continuous for any $C$. Our definition of $dd^c$ is such that $(dd^c \log^+ {\max[|z_1|,...,|z_d|]})^d$ is a probability measure.

We recall the definition of $C-$transfinite diameter $\delta_C(K)$ of a compact set $K\subset \C^d$. Letting $d_n$ be the dimension of $\Poly(nC)$, we have
$$\Poly(nC)= \hbox{span} \{e_1,...,e_{d_n}\}$$ 
where $\{e_j(z):=z^{\alpha(j)}=z_1^{\alpha_1(j)} \cdots z_d^{\alpha_d(j)}\}_{j=1,...,d_n}$ are the standard basis monomials in $\Poly(nC)$ in any order. For 
points $\zeta_1,...,\zeta_{d_n}\in \C^d$, let
$$VDM(\zeta_1,...,\zeta_{d_n}):=\det [e_i(\zeta_j)]_{i,j=1,...,d_n}  $$
$$= \det
\left[
\begin{array}{ccccc}
 e_1(\zeta_1) &e_1(\zeta_2) &\ldots  &e_1(\zeta_{d_n})\\
  \vdots  & \vdots & \ddots  & \vdots \\
e_{d_n}(\zeta_1) &e_{d_n}(\zeta_2) &\ldots  &e_{d_n}(\zeta_{d_n})
\end{array}
\right]$$
and for a compact subset $K\subset \C^d$ let
\begin{equation}\label{vn} V_n =V_n(K):=\max_{\zeta_1,...,\zeta_{d_n}\in K}|VDM(\zeta_1,...,\zeta_{d_n})|.\end{equation}
Then
\begin{equation}\label{deltac} \delta_C(K):= \limsup_{n\to \infty}V_{n}^{1/l_n}\end{equation}
is the $C-$transfinite diameter of $K$ where $l_n:=\sum_{j=1}^{d_n} {\rm deg}(e_j)$. 

The existence of the limit is not obvious. In this generality it was proved in \cite{BBL}. In the classical ($C=\Sigma$) case, Zaharjuta \cite{Zah} verified the existence of the limit by introducing directional Chebyshev constants $\tau(K,\theta)$ and proving 
$$\delta_{\Sigma}(K)=\exp \bigl(\frac{1}{|\sigma|}\int_{\sigma^0} \log \tau(K,\theta)d|\sigma|(\theta)\bigr)$$
where $\sigma:=\{(x_1,...,x_d)\in \R^d: 0\leq x_i \leq 1, \ \sum_{j=1}^d x_i = 1\}$ is the extreme ``face'' of $\Sigma$; $\sigma^0=\{(x_1,...,x_d)\in \R^d: 0 < x_i <1, \ \sum_{j=1}^d x_i = 1\}$; and $|\sigma|$ is the $(d-1)-$dimensional measure of $\sigma$.  We will utilize results from \cite{Sione} where a Zaharjuta-type proof of the existence of the limit in the general $C-$setting is given. There it is shown that
\begin{equation}\label{zahtype} \delta_C(K)=\bigl[\exp \bigl(\frac{1}{vol(C)}\int_{C^o} \log \tau_C(K,\theta)dm(\theta)\bigr) \bigr]^{1/A_C}\end{equation}
where the directional Chebyshev constants $\tau_C(K,\theta)$ and the integration in the formula are over the interior $C^o$ of the entire $d-$dimensional convex body $C$ and $A_C$ is a positive constant depending only on $C$ and $d$ (defined in (\ref{target2})). 

Apriori, in the definition of $\tau_C(K,\theta)$ the standard grlex (graded lexicographic) ordering $\prec$ on $\N^d$ (i.e., on the monomials in $\C^d$) was used. This was required to obtain the submultiplicativity of the ``monic'' polynomial classes 
\begin{equation} \label{monicclass} M_k(\alpha):= \{p\in \Poly(kC): p(z)=z^{\alpha} +\sum_{\beta\in kC\cap\N^d, \ \beta \prec \alpha} c_{\beta}z^{\beta}\}\end{equation}
for $\alpha \in kC\cap\N^d$; i.e., $M_{k_1}(\alpha_1) \cdot M_{k_2}(\alpha_2) \subset M_{k_1+k_2}(\alpha_1+\alpha_2)$. Defining Chebyshev constants
\begin{equation}\label{cheb} T_k(K,\alpha):=\inf \{\|p\|_K:p\in M_k(\alpha)\}^{1/k},\end{equation}
for $\theta \in C^{o}$, this submultiplicativity allows one to verify existence of the limit
\begin{equation}\label{dircheb} \tau_C(K,\theta):=\lim_{k\to \infty, \ \alpha/k\to \theta} T_k(K,\alpha)\end{equation}
as well as convexity of the function $\theta \to \ln \tau_C(K,\theta)$ on $C^{o}$.

In the proof that $\lim_{n\to \infty}V_{n}^{1/l_n}$ exists in \cite{Sione}, it is shown that 
\begin{equation}\label{target} \lim_{n\to \infty}V_{n}^{1/nd_n}=\lim_{n\to \infty} \bigl(\prod_{j=1}^{d_n} T_n(K,\alpha(j))^n\bigr)^{1/nd_n}.\end{equation}
The asymptotic relation between $nd_n$ and $l_n$ is that
\begin{equation}\label{target2}\lim_{n\to \infty} \frac{l_n}{nd_n} = A_C:=\frac{1}{\hbox{vol}(C)}\cdot \iint_C  (x_1+\cdots x_d)dx_1 \cdots dx_d=:M_C/\hbox{vol}(C).\end{equation}
The following propositions will be useful in the sequel.

\begin{proposition}\label{tdscale} For $t>0$,
$$\delta_{tC}(K)=\delta_C(K).$$
\end{proposition}

\begin{proof} We first observe that if $t\in \N$, since the limit in (\ref{deltac}) exists, 
$$\delta_C(K)= \lim_{n\to \infty}V_{n}^{1/l_n}= \lim_{n\to \infty}V_{tn}^{1/l_{tn}}=\delta_{tC}(K).$$
Similarly, if $t \in \Q$ we have $\delta_{tC}(K)=\delta_C(K)$. To verify the result for $t\in \R$, we proceed as follows. If 
$t_1<t<t_2$, from the definitions of $M_k(\alpha), \ T_k(K,\alpha)$ and $\tau_C(K,\theta)$, we have the following: 
\begin{enumerate}
\item for $\theta \in t_1C^o$, $\tau_{t_1C}(K,\theta)\geq  \tau_{tC}(K,\theta)$; and
\item for $\theta \in tC^o$, $\tau_{tC}(K,\theta)\geq  \tau_{tC_2}(K,\theta)$.
\end{enumerate}
Taking a sequence $\{t_{1,j}\}\subset \Q$ with $t_{1,j}\uparrow t$ and a sequence $\{t_{2,j}\}\subset \Q$ with $t_{2,j}\downarrow t$, using the above inequalities together with (\ref{zahtype}) and (\ref{target2}), 
$$\lim_{j\to \infty} \delta_{t_{1,j}C}(K) = \lim_{j\to \infty} \delta_{t_{2,j}C}(K) =\delta_{tC}(K).$$
\end{proof}

We can use the Hausdorff metric on the family of our convex bodies $C$ satisfying (\ref{stdhyp}) considered as compact sets in $\R^d$. Using similar ideas from the previous proof, we verify the next result.

\begin{proposition}\label{tdcont} Given $K\subset \C^d$, the mapping $C\to \delta_C(K)$ is continuous.
\end{proposition}

\begin{proof} Taking a sequence $\{C_j\}$ of convex bodies satisfying (\ref{stdhyp}) converging to $C$ in the Hausdorff metric, we can find $\epsilon_j \to 0$ with 
$$(1-\epsilon_j)C \subset C_j \subset (1+\epsilon_j)C, \ j=1,2,...$$
As in the proof of Proposition \ref{tdscale}, we have
\begin{enumerate}
\item for $\theta \in (1-\epsilon_j)C^o$, $\tau_{(1-\epsilon_j)C}(K,\theta)\geq  \tau_{C_j}(K,\theta)$; and
\item for $\theta \in C_j^o$, $\tau_{C_j}(K,\theta)\geq  \tau_{(1+\epsilon_j)C}(K,\theta)$.
\end{enumerate}
Since $\epsilon_j \to 0$ implies $\hbox{vol}(C_j)\to \hbox{vol}(C)$ and $M_{C_j}\to M_C$, using the above inequalities together with (\ref{zahtype}) and (\ref{target2}), we find $\lim_{j\to \infty} \delta_{C_j}(K)=\delta_C(K)$.
\end{proof}

For most of the subsequent sections, we work in $\C^2$. First, recall the definition of the Robin function ${\bf \rho_u}$ associated to $u\in L(\C^2)$: 
$${\bf \rho_u}(z):=\limsup_{|\lambda|\to \infty} [u(\lambda z)-\log |\lambda |]. $$
For $z=(z_1,z_2)\not = (0,0)$ we define
$$ \underline {\bf \rho_u}(z):=\limsup_{|\lambda|\to \infty} [u(\lambda z)-\log |\lambda z|]={\bf \rho_u}(z)-\log |z| $$
so that $\underline {\bf \rho_u}(tz)=\underline {\bf \rho_u}(z)$ for $t\in \C \setminus \{0\}$. Here $|z|^2=|z_1|^2+|z_2|^2$. We can consider $\underline {\bf \rho_u}$ as a function on $\P^1=\P^2\setminus \C^2$ where to $p=(p_1,p_2)$ with $|p|=1$ we associate the point where the complex line $\lambda \to \lambda p$ hits $\P^1$. 

For a special class of convex bodies, there is a generalization of the notion of Robin function. Following \cite{LM}, if we let $C$ be the triangle $T_{a,b}$ with vertices $(0,0), (b,0), (0,a)$ where $a,b$ are relatively prime positive integers, we have the following:
\begin{enumerate}
\item $H_C(z_1,z_2)=\max[\log^+|z_1|^b, \log^+|z_2|^a]$ (note $H_C=0$ on the closure of the unit polydisk $P^2:=\{(z_1,z_2): |z_1|,|z_2|< 1\}$), and, indeed, $H_C=V_{C,P^2}=V_{C,T^2}$ where $T^2:=\{(z_1,z_2): |z_1|,|z_2|= 1\}$;
\item defining $\lambda \circ (z_1,z_2):=(\lambda^az_1,\lambda^b z_2)$, we have 
$$ H_C(\lambda \circ (z_1,z_2))=H_C(z_1,z_2)+ab\log|\lambda| $$ for $(z_1,z_2)\in \C^2 \setminus P^2$ and $|\lambda| \geq 1$.
\end{enumerate}

\begin{definition} \label{crobtri} Given $u\in L_C$, we define the $C-$Robin function of $u$:  
$$\rho_u(z_1,z_2):=\limsup_{|\lambda|\to \infty} [u(\lambda \circ (z_1,z_2))-ab\log|\lambda|]$$
for $(z_1,z_2)\in \C^2$.
\end{definition}

Applying the transformation formula Theorem 4.1 of \cite{LM} in the case where $d=2$; $C$ is our triangle with vertices $(0,0), (b,0), (0,a)$;  $C'=ab\Sigma$; and we consider the proper polynomial mapping
$$F(z_1,z_2)=(z_1^a,z_2^b),$$
we obtain
$$abV_{F^{-1}(K)}(z_1,z_2)=V_{C,K}(z_1^a,z_2^b)$$ 
so that
$$ab {\bf \rho_{V_{F^{-1}(K)}}}(z_1,z_2)=\limsup_{|\lambda|\to \infty} [V_{C,K}(\lambda^az_1^a,\lambda^bz_2^b)-ab\log |\lambda|]$$
$$=\limsup_{|\lambda|\to \infty} [V_{C,K}(\lambda \circ (z_1^a,z_2^b))-ab\log |\lambda|]$$
$$=\rho_{V_{C,K}}(z_1^a,z_2^b)= \rho_{V_{C,K}}(F(z_1,z_2)).$$

More generally, letting 
$$ \zeta=(\zeta_1,\zeta_2)=F(z)=F(z_1,z_2)=(z_1^a,z_2^b), $$ 
for $u\in L_C$, we have 
\begin{equation}\label{raheq} \tilde u(z):=u(F(z_1,z_2))=u(\zeta) \in abL \ \hbox{and}\end{equation}
\begin{equation}\label{roheq} \rho_u(\zeta)=\rho_u(F(z_1,z_2))=ab{\bf \rho_{\tilde u/ab}}(z)\end{equation} 
where ${\bf \rho_{\tilde u/ab}}$ is the standard Robin function of $\tilde u/ab\in L$. Note that if $u\in L_C^+$ then $\tilde u \in abL^+$. We apply these results in the next section.

\section{$C-$Rumely formula for $C=T_{a,b}$} In this section, we let $C=T_{a,b}$. We begin with some integral formulas associated to functions in $L^+(\C^2)$. The integral formula Theorem 5.5 of \cite{BT} in this setting is the following.

\begin{theorem} \label{bt5.5} {\bf(Bedford-Taylor)} Let $u,v,w\in L^+(\C^2)$. Then
$$\int_{\C^2} (udd^cv-vdd^cu)\wedge dd^cw=   \int_{\P^1} (\underline {\bf \rho_u}-\underline {\bf \rho_v}) (dd^c \underline{\bf \rho_w}+\omega)$$
where $\omega$ is the standard K\"ahler form on $\P^1$.
\end{theorem}

Next, following the arguments in \cite{DR}, we get a symmetrized integral formula involving Robin functions ${\bf \rho_u}, {\bf \rho_v}$ for $u,v \in L^+(\C^2)$ and their projectivized versions $ \underline {\bf \rho_u}, \underline {\bf \rho_v}$:
\begin{equation}\label{unproj} \int_{\P^1}(\underline{\bf \rho_u}-\underline{\bf \rho_v})\bigl[(dd^c \underline{\bf \rho_u}+\omega)+(dd^c \underline{\bf \rho_v}+\omega)]\end{equation}
$$=\int_{\C^1}{\bf \rho_u}(1,t)dd^c {\bf \rho_u}(1,t) +{\bf \rho_u}(0,1)-[\int_{\C^1}{\bf \rho_v}(1,t)dd^c {\bf \rho_v}(1,t) +{\bf \rho_v}(0,1)].$$
From (\ref{raheq}), if $u,v,w\in L_C^+$,
$$ab\int_{\C^2} (udd^cv-vdd^cu)\wedge dd^cw = \int_{\C^2} (\tilde udd^c \tilde v-\tilde vdd^c\tilde u)\wedge dd^c\tilde w.$$
We apply Theorem \ref{bt5.5}  to the right-hand-side, multiplying by factors of $ab$ since $\tilde u, \tilde v, \tilde w\in abL^+$, to obtain, with the aid of (\ref{roheq}), the desired integral formula (cf., (6.3) in \cite{LM}):
\begin{equation}\label{riheq} \int_{\C^2} (udd^cv-vdd^cu)\wedge dd^cw =  (ab)^2 \int_{\P^1} (\underline {\bf \rho_{\tilde u/ab}}-\underline {\bf \rho_{\tilde v/ab}}) (dd^c \underline {\bf \rho_{\tilde w/ab}}+ \omega).\end{equation}

Next, for $u,v \in L_C^+$, we define the mutual energy
$$
\mathcal E (u,v):= \int_{\C^2} (u-v)[(dd^cu)^2+dd^cu\wedge dd^cv+(dd^cv)^2].
$$
 (cf., (3.1) in \cite{BBL}). We connect this notion with $C-$transfinite diameter by recalling the following formula from \cite{BBL}.
\begin{theorem} \label{energyrumely} Let $K\subset \C^2$ be compact and nonpluripolar. Then
$$
\log \delta_C(K)=  \frac{-1}{c}\mathcal E(V_{C,K}^*,H_C)
$$
where $c=3!M_C$ with $M_C:=\iint_C (x+y)dx dy$.
\end{theorem}

\begin{remark} This formula is actually valid in $\C^d$ for $d>1$ for any convex body $C\subset (\R^+)^d$ satisfying (\ref{stdhyp}) with the appropriate definitions of $\mathcal E$ and $c$. 
\end{remark}

Our goal in this section is to rewrite $\mathcal E(V_{C,K}^*,H_C)$ using the integral formulas in order to get a formula relating $\delta_C(K)$ and ${\bf \rho_{\tilde V_{C,K}/ab}}$ more in the spirit of Proposition 3.1 in \cite{DR}. This will be used in the next section to prove a formula for the $C-$transfinite diameter $\delta_C(K)$ of a product set $K=E\times F$.

\begin{proposition} \label{prop23} We have
$$\mathcal E(V_{C,K}^*,H_C)=(ab)^2 [\int_{\C^1} {\bf \rho_{\tilde V_{C,K}/ab}}(1,t)dd^c{\bf \rho_{\tilde V_{C,K}/ab}}(1,t)-{\bf \rho_{\tilde V_{C,K}/ab}}(0,1)].$$
Hence from Theorem \ref{energyrumely}
\begin{equation}\label{rumelytrue}
-3!M_C\log \delta_C(K)= (ab)^2 [\int_{\C^1} {\bf \rho_{\tilde V_{C,K}/ab}}(1,t)dd^c{\bf \rho_{\tilde V_{C,K}/ab}}(1,t)-{\bf \rho_{\tilde V_{C,K}/ab}}(0,1)].\end{equation}

\end{proposition}

\begin{proof} Applying the formula (\ref{riheq}) with $w=u$ and $w=v$ and adding, we obtain
$$ \int_{\C^2} (udd^cv-vdd^cu)\wedge dd^c(u+v) =  (ab)^2 \int_{\P^1} (\underline {\bf \rho_{\tilde u/ab}}-\underline {\bf \rho_{\tilde v/ab}}) [(dd^c \underline {\bf \rho_{\tilde u/ab}}+ \omega)+(dd^c \underline {\bf \rho_{\tilde v/ab}}+ \omega)].$$
We claim from the definition of $\mathcal E(u,v)$, it follows that
\begin{equation}\label{ibyp}\mathcal E(u,v)=  \int_{\C^2} [u(dd^cv)^2-v(dd^cu)^2]+  (ab)^2 \int_{\P^1} (\underline {\bf \rho_{\tilde u/ab}}-\underline {\bf \rho_{\tilde v/ab}}) [(dd^c \underline {\bf \rho_{\tilde u/ab}}+ \omega)+(dd^c \underline {\bf \rho_{\tilde v/ab}}+ \omega)].\end{equation}
To see this, using the previous formula it suffices to show 
$$\mathcal E(u,v)- \int_{\C^2} u(dd^cu)^2 +\int_{\C^2}v(dd^cv)^2= \int_{\C^2} (udd^cv-vdd^cu)\wedge dd^c(u+v).$$
In verifying this, all integrals are over $\C^2$. We write
$$\mathcal E(u,v)=\int (u-v)[(dd^cu)^2+(dd^cv\wedge dd^c(u+v)]$$
$$=\int u(dd^cu)^2-\int v(dd^cu)^2 +\int(u-v)dd^cv\wedge dd^c(u+v)$$
$$=\int u(dd^cu)^2-\int v(dd^cv)^2 +\int(u-v)dd^cv\wedge dd^c(u+v)+\int v[(dd^cv)^2-(dd^cu)^2].$$
We finish this proof by working with the sum of the last two integrals:
$$\int(u-v)dd^cv\wedge dd^c(u+v)+\int v[(dd^cv)^2-(dd^cu)^2]$$
$$=\int(u-v)dd^cv\wedge dd^c(u+v)+\int v[dd^c(v-u)\wedge dd^c(u+v)]$$
$$=\int(udd^cv-vdd^cu)\wedge dd^c(u+v)$$
as desired.

Letting $u=V_{C,K_1}$ and $v=V_{C,K_2}$ in (\ref{ibyp}) where $K_1,K_2$ are regular compact sets in $\C^2$, 
$$\mathcal E(V_{C,K_1},V_{C,K_2})=   (ab)^2 \int_{\P^1} (\underline {\bf \rho_{\tilde V_{C,K_1}/ab}}-\underline {\bf \rho_{\tilde V_{C,K_2}/ab}}) [(dd^c \underline {\bf \rho_{\tilde V_{C,K_1}/ab}}+ \omega)+(dd^c \underline {\bf \rho_{\tilde V_{C,K_2}/ab}}+ \omega)].$$
In particular, since $H_C = V_{C,P^2}=V_{C,T^2}$ where $T^2$ is the unit torus in $\C^2$, 
$$\mathcal E(V_{C,K_1},H_C)=   (ab)^2 \int_{\P^1} (\underline {\bf \rho_{\tilde V_{C,K_1}/ab}}-\underline {\bf \rho_{\tilde H_C/ab}}) [(dd^c \underline {\bf \rho_{\tilde V_{C,K_1}/ab}}+ \omega)+(dd^c \underline {\bf \rho_{\tilde H_C/ab}}+ \omega)].$$

The result will follow from (\ref{unproj}) once we verify
\begin{equation}\label{hccalc}
\int_{\C^1} {\bf \rho_{\tilde H_C/ab}}(1,t)dd^c{\bf \rho_{\tilde H_C/ab}}(1,t)+{\bf \rho_{\tilde H_C/ab}}(0,1)=0.
\end{equation}
To verify (\ref{hccalc}), we begin by observing that since
$$H_C(z_1,z_2)=\max\left(\log^+|z_1|^b,\log^+|z_2|^a\right), \ \tilde H_C(z_1,z_2) := H_C(z_1^a,z_2^b),$$
for $(z_1,z_2)\in \C^2\setminus (P^2)^o$, 
$$\rho_{H_C}(z_1^a,z_2^b)=H_C(z_1^a,z_2^b)=ab{\bf \rho_{\tilde H_C/ab}}(z_1,z_2).$$
In particular,
$${\bf \rho_{\tilde H_C/ab}}(0,1)=\frac{1}{ab}H_C(0,1)=0 \ \hbox{and}$$
$${\bf \rho_{\tilde H_C/ab}}(1,t)=\frac{1}{ab}H_C(1,t^b)=\frac{1}{ab}\max\left(0,ab\log|t|\right).$$
Thus $dd^c{\bf \rho_{\tilde H_C/ab}}(1,t)$ is supported on $|t|=1$ where it is (normalized) arclength measure. On this set, we have ${\bf \rho_{\tilde H_C/ab}}(1,t)=0$ and (\ref{hccalc}) follows.
\end{proof}

\section{Product formula for $C$ a triangle}  We first use (\ref{rumelytrue}) to prove a formula for the $C-$transfinite diameter of a product set when $C$ is a triangle $T_{a,b}$ with vertices $(0,0), (b,0), (0,a)$. Then in the next section we give a (conceptually) simpler proof that is valid for general convex bodies. 

\begin{theorem} \label{triprod} Let $K=E\times F$ where $E,F\subset \C$ are compact. Then for $C=T_{a,b}$,  
$$-\log \delta_C(K)=\frac{ab}{a+b}\cdot \left( \frac{-\log D(E)}{a}+ \frac{-\log D(F)}{b}\right); \ \hbox{i.e.,}$$
$$\delta_C(K)=D(E)^{{b}/(a+b)}D(F)^{{a}/(a+b)}$$
where $D(E), D(F)$ are the univariate transfinite diameters of $E,F$. 

\end{theorem}

\begin{proof} We first assume $a,b$ are positive integers and use Proposition \ref{prop23}. To this end, we compute ${\bf \rho_{\tilde V_{C,K}/ab}}$ for $K=E\times F$. We can assume $E,F$ are regular compact sets in $\C$ and we let $\rho_E=-\log D(E)$ and $\rho_F=-\log D(F)$ be the Robin constants of these sets. From Proposition 2.4 of \cite{BosLev}, 
$$V_{C,K}(z_1,z_2)=\max\left(bg_E(z_1),ag_F(z_2)\right)$$
where $g_E,g_F$ are the Green functions for $E,F$. Note that 
$$\rho_E =\lim_{|z_1|\to \infty}[g_E(z_1)-\log |z_1|] \ \hbox{and} \ \rho_F =\lim_{|z_2|\to \infty}[g_F(z_2)-\log |z_2|].$$
Thus from Definition \ref{crobtri}
$$ \rho_{V_{C,K}}(z_1,z_2)=\limsup_{|\lambda|\to \infty}[\max[bg_E(\lambda^az_1),ag_F(\lambda^bz_2)]-ab\log|\lambda|]$$
$$=\max[b(\rho_E+\log|z_1|),a(\rho_F+\log|z_2|)]$$
so that
$${\bf \rho_{\tilde V_{C,K}/ab}}(z_1,z_2)= \frac{1}{ab}\rho_{V_{C,K}}(z_1^a,z_2^b)=\max[\frac{1}{a}\rho_E+\log|z_1|,\frac{1}{b}\rho_F+\log|z_2|].$$

Hence
$${\bf \rho_{\tilde V_{C,K}/ab}}(1,t)=\max\left(\frac{1}{a}\rho_E,\frac{1}{b}\rho_F+\log|t|\right)$$
so that $dd^c{\bf \rho_{\tilde V_{C,K}/ab}}(1,t)$ is normalized arclength meaure on a circle where the value of the function ${\bf \rho_{\tilde V_{C,K}/ab}}(1,t)=\frac{1}{a}\rho_E$. Finally, ${\bf \rho_{\tilde V_{C,K}/ab}}(0,1)=\frac{1}{b}\rho_F$ and the result when $a,b$ are positive integers follows from Proposition \ref{prop23} since 
$$(ab)^2 [\int_{\C^1} {\bf \rho_{\tilde V_{C,K}/ab}}(1,t)dd^c{\bf \rho_{\tilde V_{C,K}/ab}}(1,t)-{\bf \rho_{\tilde V_{C,K}/ab}}(0,1)]=(ab)^2\left(\frac{1}{a}\rho_E+\frac{1}{b}\rho_F\right)$$
and a calculation shows that $M_C=(ab/6)(a+b)$ so that $3!M_C=(ab)(a+b)$. If $a,b\in \Q$, the result follows from Proposition \ref{tdscale}; finally, the general case when $a,b\in \R$ follows from Proposition \ref{tdcont}.
\end{proof}

\section{Product formula for general $C$} In this section, we give an alternate proof of Theorem \ref{triprod} which is applicable in a much more general setting. We assume that $C$ is a convex body satisfying (\ref{stdhyp}) which is a {\it lower set}: whenever $(j_1,j_2)\in nC\cap \N^2$ we have $(k_1,k_2)\in nC\cap \N^2$ for all $k_l\leq j_l, \ l=1,2$. For example, the triangles $T_{a,b}$ are lower sets. This proof is modeled on that of Bloom-Calvi in \cite{BC}. As in the previous section, we take $K=E\times F$ where $E,F$ are compact sets in $\C$. Let $\mu_E,\mu_F$ be {\it Bernstein-Markov measures} for $E,F$: recall $\nu$ is a Bernstein-Markov measure for $E$ if for any $\epsilon >0$, there exists a constant $c_{\epsilon}$ so that
$$\|p_n\|_K\leq c_{\epsilon}(1+\epsilon)^n \|p_n\|_{L^2(\nu)}, \ n=1,2,...$$
where $p_n$ is any polynomial of degree $n$. If $E,F$ are regular, one can take, e.g., $\mu_E$ and $\mu_F$ to be the distributional Laplacians of the Green functions $g_E$ and $g_F$. Let $\mu:=\mu_E \otimes \mu_F$. Let $\{p_j(z)\}_{j=0,1,2,...}$ be monic orthogonal polynomials for $L^2(\mu_E)$ and let $\{q_k(z)\}_{k=0,1,2,...}$ be monic orthogonal polynomials for $L^2(\mu_F)$; then $\{p_j(z)q_k(w)\}_{j,k=0,1,2,...}$ are orthogonal in $L^2(\mu)$. Using the grlex ordering $\prec$ on $\N^2$ and the lower set property of $C$, it is easy to see that each $L^2(\mu)-$orthogonal polynomial $p_j(z)q_k(w)$ is in a class $M_l(\alpha)$ (recall (\ref{monicclass})) where $\alpha =(j,k)$ and $l=\deg_C(z^jw^k)$. Here and below $j,k$ are nonnegative integers. 

We want to use (\ref{target}): the asymptotics of $V_n$ and $\prod_{j=1}^{d_n} T_n(K,\alpha(j))^n$ are the same; i.e., the limits of their $nd_n-$th roots coincide. If $\mu$ is a Bernstein-Markov measure on $K$, it follows readily that one can replace the sup-norm minimizers $T_k(K,\alpha)$ by $L^2(\mu)-$norm minimizers
$$\tilde T_k(K,\alpha):=\inf \{\|p\|_{L^2(\mu)}:p\in M_k(\alpha)\}^{1/k}.$$
In our setting, for $\alpha =(j,k)$ the polynomial $p_j(z)q_k(w)$ is the minimizer and
$$\|p_jq_k\|_{L^2(\mu)}=\|p_j\|_{L^2(\mu_E)}\cdot \|q_k\|_{L^2(\mu_F)}.$$ 
Moreover, we know from the univariate theory that
$$\lim_{j\to \infty} \|p_j\|_{L^2(\mu_E)}^{1/j}=D(E) \ \hbox{and} \ \lim_{k\to \infty} \|q_k\|_{L^2(\mu_F)}^{1/k}=D(F).$$ 
For simplicity, we write
$$p_j:=\|p_j\|_{L^2(\mu_E)} \ \hbox{and} \ q_k:=\|q_k\|_{L^2(\mu_F)}.$$
In this notation, to utilize (\ref{target}), we consider
$$\left(\prod_{(j,k)\in nC} p_jq_k\right)^{1/nd_n}.$$

We suppose that $(b,0)$ and $(0,a)$ are extreme points of $C$ and that the outer face $F_C$ of $C$; i.e., the portion of the topological boundary of $C$ outside of the coordinate axes, can be written both as a graph $\{(x,f(x)):0\leq x \leq b\}$ and as a graph $\{(g(y),y):0\leq y\leq a\}$. 

\begin{theorem}\label{genprod} Let $K=E\times F$ where $E,F$ are compact subsets of $\C$. Then 
\begin{equation}\label{abform} \delta_C(K)=D(E)^{{A}/{(A+B)}}\cdot D(F)^{{B}/{(A+B)}}\end{equation}
where $A=\int_0^buf(u)du$ and $B=\int_0^aug(u)du$. Hence for any convex body $C$ with $A=B$ we obtain 
$$\delta_C(K)=[D(E)D(F)]^{1/2}.$$
\end{theorem}

\begin{proof} The outer face $F_{nC}$ of $nC$ can be written as 
$$\{(x,y):y=f_n(x):=nf(x/n), \ 0\leq x \leq nb\}=\{(x,y):x=g_n(y):=ng(y/n), \ 0\leq y \leq na\}.$$
Then the product $\prod_{(j,k)\in nC} p_jq_k;$ i.e., the product over the integer lattice points in $nC$, is asymptotically given by 
$$p_1^{f_n(1)}p_2^{f_n(2)} \cdots p_{nb}^{f_n(nb)}q_1^{g_n(1)}q_2^{g_n(2)}\cdots q_{na}^{g_n(na)}$$
$$=\left(p_1^{f(1/n)}p_2^{f(2/n)} \cdots p_{nb}^{f(nb/n)}q_1^{g(1/n)}q_2^{g(2/n)}\cdots q_{na}^{g(na/n)}\right)^n.$$
For simplicity in the calculation, we concentrate on the product 
$$p_1^{f(1/n)}p_2^{f(2/n)} \cdots p_{nb}^{f(nb/n)}$$ 
Then using the fact that $p_j\asymp D(E)^j$,
$$p_1^{f(1/n)}p_2^{f(2/n)} \cdots p_{nb}^{f(nb/n)} \asymp D(E)^{f(1/n)+2f(2/n)+\cdots + nbf(nb/n)}$$
$$=D(E)^{n^2\cdot (1/n)[1/nf(1/n)+2/nf(2/n)+\cdots + nb/nf(nb/n)]}\asymp D(E)^{n^2\int_0^b uf(u)du}.$$
Similarly, since $q_k\asymp D(F)^k$, we have
$$q_1^{g(1/n)}q_2^{g(2/n)}\cdots q_{na}^{g(na/n)}\asymp D(F)^{n^2\int_0^a ug(u)du}.$$
Hence
$$\prod_{(j,k)\in nC} p_jq_k \asymp D(E)^{n^3\int_0^b uf(u)du}\cdot D(F)^{n^3 \int_0^a ug(u)du}.$$
Using (\ref{target}) and (\ref{target2}), since 
$$nd_nA_C\asymp n n^2 \hbox{area}(C) \cdot \frac{M_C}{\hbox{area}(C)}=n^3M_C,$$
and
$$M_C=\iint_C xdy dx +\iint_C y dxdy = \int_0^b \int_0^{f(x)}x dy dx +\int_0^a\int_0^{g(y)}y dx dy$$
$$=\int_0^b x f(x) dx +\int_0^a yg(y)dy =A+B,$$
(\ref{abform}) follows.
\end{proof}

\begin{remark} \label{onehalf} Note that $A=B$ occurs whenever $a=b$ and $f=g$; i.e., the convex body is symmetric about the line $y=x$. As special cases of this, we can take 
\begin{equation}\label{cpdef} C=C_p:=\{(x,y):x,y\geq 0, \ x^p +y^p \leq 1\}, \ 1\leq p <\infty \end{equation}
as well as 
$$C_{\infty}:= \{(x,y):0\leq x,y\leq 1\}.$$

\end{remark}

\begin{remark} We can use Theorem \ref{genprod} to verify our result in Theorem \ref{triprod} for $C=T_{a,b}$. Here 
$y=f(x)=a(1-x/b), \ 0\leq x \leq b$ and $x=g(y)= b(1-y/a), \ 0\leq y\leq a$. Then
$$\int_0^b xf(x)dx =\frac{ab^2}{6}; \ \int_0^a yg(y)dy = \frac{ba^2}{6};$$
and
$$M_{T_{a,b}}=\int_0^b \int_0^{a(1-x/b)}(x+y)dydx= (ab/6)(a+b).$$
Hence
$$\delta_{T_{a,b}}(K)=D(E)^{{b}/{(a+b)}}D(F)^{{a}/{(a+b)}}.$$
Moreover, the calculations in Theorem \ref{genprod} -- and the resulting formula -- are valid (and much easier) in a special case where $F_C$ cannot be written as a graph $\{(x,f(x)):0\leq x \leq b\}$ (nor as a graph $\{(g(y),y):0\leq y\leq a\}$); namely, the rectangle $C=R_{a,b}$ with vertices $(0,0),(b,0),(0,a)$ and $(b,a)$. Here we take $y=f(x)=a, \ 0\leq x\leq b$ and $x=g(y)=b, \ 0\leq y\leq a$ and the calculations in the proof of Theorem \ref{genprod} yield
$$\int_0^b xf(x)dx =\frac{ab^2}{2}; \ \int_0^ax yg(y)dy = \frac{ba^2}{2};$$
and
$$M_{R_{a,b}}=\int_0^b \int_0^{a}(x+y)dydx= (ab/2)(a+b).$$
Hence, for this rectangle we recover the same product formula as for $T_{a,b}$:
$$\delta_{R_{a,b}}(K)=D(E)^{{b}/{(a+b)}}D(F)^{{a}/{(a+b)}}.$$

\end{remark}

\section{The case of $d-$circled sets} One might wonder, given Remark \ref{onehalf}, whether we {\it always} have equality of $\delta_C(K)$ for all convex bodies $C$ that are symmetric about the line $y=x$ (e.g., $C_p$ for $1\leq p\leq \infty$), i.e., for {\it any} compact set $K$, not just product sets. This is not the case as we will illustrate for  
$$\B:=\{(z_1,z_2):|z_2|^2+|z_2|^2\leq 1\},$$
the closed Euclidean unit ball in $\C^2$. This is an example of a {\it $2-$circled} set. We say a set $E \subset \C^d$ is {\it $d-$circled} if 
$$
(z_1,...,z_d)\in E\text{ implies }(e^{i\beta_1}z_1,...,e^{i\beta_d}z_d)\in E,\text{ for all real }\beta_1,...,\beta_d.$$ 
For a compact, $d-$circled set $K$, it is easy to see from the Cauchy estimates that 
$$\inf \{\|p\|_K:p\in M_k(\alpha)\}=\|z^{\alpha}\|_K$$
where recall 
$$M_k(\alpha):= \{p\in \Poly(kC): p(z)=z^{\alpha} +\sum_{\beta\in kC\cap\N^d, \ \beta \prec \alpha} c_{\beta}z^{\beta}\}.$$
Then for any convex body $C$ satisfying (\ref{stdhyp}), and any $\theta =(\theta_1,...,\theta_d) \in C^{o}$, we have 
$$\tau_C(K,\theta):=\lim_{k\to \infty, \ \alpha/k\to \theta} T_k(K,\alpha)
=\lim_{k\to \infty, \ \alpha/k\to \theta}\|z^{\alpha}\|_K^{1/k}=\max_{z\in K} |z_1|^{\theta_1}\cdots |z_d|^{\theta_d}.$$
Thus, for a given $d-$circled set $K$, if we can explicitly determine these values, we can use (\ref{zahtype}) to compute $\delta_C(K)$. 

Indeed, an elementary calculation for $K=\B\subset \C^2$ shows that 
\begin{equation}\label{balltau}\tau_C(\B,\theta)=\left(\frac{\theta_1}{\theta_1+\theta_2}\right)^{\theta_1/2}\left(\frac{\theta_2}{\theta_1+\theta_2}\right)^{\theta_2/2}.
\end{equation}
It follows readily from (\ref{zahtype}) that $\delta_{C_1}(\B)=e^{-1/4}$.

We next show that the main result in \cite{Sione}, specifically, equation (\ref{zahtype}) in our Section 1, remains valid even for certain {\it nonconvex sets} $C$ and all $d-$circled sets $K$. To this end, let $C\subset (\R^+)^d$ be the closure of an open, connected set satisfying (\ref{stdhyp}). As examples, one can take $C_p$ as in (\ref{cpdef}) for $0<p<1$. Here, the definition of 
\begin{align*}
\Poly(nC)=\{p(z) & =\sum_{J\in nC \cap \N^d}c_{J}z^{J},~c_{J}\in\C\},\quad n=1,2,\ldots
\end{align*}
makes sense; and we have $\Poly(nC)= \hbox{span} \{e_1,...,e_{d_n}\}$ 
where $e_j(z):=z^{\alpha(j)}$ are the standard basis monomials in $\Poly(nC)$ and $d_n$ is the dimension of $\Poly(nC)$. Using the same notation 
$$VDM(\zeta_1,...,\zeta_{d_n}):=\det [e_i(\zeta_j)]_{i,j=1,...,d_n}  $$
as in the convex setting, for a compact and $d-$circled set $K\subset \C^d$, we have the same notions of maximal Vandermonde $V_n=V_n(K)$; $C-$transfinite diameter $\delta_C(K)$; ``monic'' polynomial classes $M_k(\alpha)$ and corresponding Chebyshev constants $T_k(K,\alpha)$; and directional Chebyshev constants $\tau_C(K,\theta)$ for $\theta \in C^{o}$ as in (\ref{vn}), (\ref{deltac}), (\ref{monicclass}), (\ref{cheb}) and (\ref{dircheb}).

\begin{proposition} \label{seven1} For $C\in (\R^+)^d$ the closure of an open, connected set satisfying (\ref{stdhyp}) and for any $d-$circled set $K\subset \C^d$, we have: 
\begin{enumerate}
\item for $\theta \in C^{o}$, 
$$\tau_C(K,\theta):=\lim_{k\to \infty, \ \alpha/k\to \theta} T_k(K,\alpha),$$
i.e., the limit exists; and 
\item $ \lim_{n\to \infty}V_{n}^{1/nd_n}$ exists and equals $\lim_{n\to \infty}[\prod_{j=1}^{d_n}T_n(K,\alpha(j))^n]^{1/nd_n}$;
\item $\delta_{C}(K)=\bigl[\exp \bigl(\frac{1}{vol(C)}\int_{C^o} \log \tau_C(K,\theta)dm(\theta)\bigr) \bigr]^{1/A_C}$ where $A_C$ is a positive constant defined in (\ref{target2}).
\end{enumerate} 

\end{proposition}

\begin{proof} Because $\inf \{\|p\|_K:p\in M_k(\alpha)\}=\|z^{\alpha}\|_K$, all the arguments in Lemmas 4.4 and 4.5 of \cite{Sione} work to show
$$\prod_{j=1}^{d_n}T_n(K,\alpha(j))^n\leq  V_n  \leq d_n! \cdot \prod_{j=1}^{d_n}T_n(K,\alpha(j))^n.$$
The only other ingredients needed to complete the rest of the proof are simply to observe that even though the polynomial classes $M_k(\alpha)$ are not submultiplicative, the {\it monomials} $z^{\alpha}$ themselves are; i.e., $z^{\alpha}z^{\beta}=z^{\alpha +\beta} \in M_k(\alpha+\beta)$. This is all that is needed to show 1.; then the proof in \cite{Sione} gives 2. and 3.

\end{proof}

From the general formula 
$$\tau_C(K,\theta)=\max_{z\in K} |z_1|^{\theta_1}\cdots |z_d|^{\theta_d}$$
for any $d-$circled set $K\subset \C^d$, 
$$\delta_{C}(K)=\left(\exp \left(\frac{1}{vol(C)}\int_{C^o} \log {\bigl(\max_{z\in K} |z_1|^{\theta_1}\cdots |z_d|^{\theta_d}\bigr)}dm(\theta)\right) \right)^{1/A_C}.$$
Using (\ref{balltau}), for $K=\B\subset \C^2$, 
\begin{equation}\label{ball2} 
\delta_{C}(\B)=
\left(\exp \left(\frac{1}{\area(C)}\int_{C^o} \log {\left(\left(\frac{\theta_1}{\theta_1+\theta_2}\right)^{\theta_1/2}\left(\frac{\theta_2}{\theta_1+\theta_2}\right)^{\theta_2/2}\right)}dm(\theta)\right) \right)^{1/A_C}.
\end{equation}
Note that for $C=C_p$ this gives a formula for the $C_{p}$-transfinite diameter of the ball $\B$ in $\C^{2}$ valid for all $0<p\leq \infty$. We return to this approach to computing $C-$transfinite diameter using directional Chebyshev constants in Proposition \ref{sixsix}. 

We can also use orthogonal polynomials as in Section 5 to compute $\delta_{C}(\B)$ for general $C$ as in Proposition \ref{seven1}; this we do next.

\begin{proposition} 
For $C$ as in Proposition \ref{seven1}, 
the $C$-transfinite diameter of the ball $\B$ in $\C^{2}$ is equal to
\begin{equation}\label{val-delt}
\delta_{C}(\B)=\exp\left(\frac{1}{2I_{4}}\left(I_{1}+I_{2}-I_{3}-\frac{\log2\pi}{2}\area(C)\right)\right),
\end{equation}
where 
\begin{equation}\label{int1-2}
I_{1}=\iiint_{C\times[0,1]}\log\Gamma(x+z)dxdydz,\qquad I_{2}=\iiint_{C\times[0,1]}\log\Gamma(y+z)dxdydz, 
\end{equation}
and
\begin{equation}\label{int3}
I_{3}=\iiint_{C\times[0,1]}\log\Gamma(x+y+z)dxdydz,\qquad I_{4}=M_C=\iint_{C}(x+y)dxdy.
\end{equation}
\end{proposition}  

\begin{proof} 
Let $\mu$ be normalized surface area on $\partial B$. Then the monomials $z^aw^b$, $a,b$ nonnegative integers, are orthogonal and 
$$\|z^aw^b\|_{L^2(\mu)}^{2}=\frac{a!b!}{(a+b+1)!},$$
see \cite[Propositions 1.4.8 and 1.4.9]{Rud}.
Let us estimate
$$
\QQ_{n}=\log\prod_{(a,b)\in nC}\frac{a!b!}{(1+a+b)!}.
$$
We have
$$
\log\prod_{(a,b)\in nC}a!=\sum_{(a,b)\in nC}\log\Gamma(a+1).
$$
Recall the multiplication formula for the Gamma function. For $\Re(z)>0$, we have
$$\Gamma (nz)=(2\pi )^{(1-n)/{2}}n^{(2nz-1)/{2}}\;\Gamma (z)\;\Gamma \left(z+{\frac {1}{n}}\right)\Gamma \left(z+{\frac {2}{n}}\right)\cdots \Gamma \left(z+{\frac {n-1}{n}}\right).$$
Applying the formula with $z=(a+1)/n$, we get
$$
\log\prod_{(a,b)\in nC}a!=\sum_{(a,b)\in nC}\sum_{k=1}^{n}\log\Gamma\left(\frac{a}{n}
+\frac{k}{n}\right)+\sum_{(a,b)\in nC}\frac{1-n}{2}\log2\pi+\sum_{(a,b)\in nC}\frac{2a-1}{2}\log n.
$$
Recalling that $d_n$, the number of elements of $nC\cap\N^{2}$, is the dimension of $\Poly(nC)$, 
\begin{multline*}
\log\prod_{(a,b)\in nC}a!=\sum_{(a,b)\in nC}\sum_{k=1}^{n}\log\Gamma\left(\frac{a}{n}
+\frac{k}{n}\right)
\\
-nd_n\frac{\log2\pi}{2}+\frac{d_n}{2}\log2\pi
+n\log n\sum_{(a,b)\in nC}\left(\frac{a}{n}\right)
-\frac{d_n}{2}{\log n}.
\end{multline*}
Interpreting the sums over the pairs $(a,b)$ as Riemann sums, we get
$$
\sum_{(a,b)\in nC}\sum_{k=1}^{n}\log\Gamma\left(\frac{a}{n}
+\frac{k}{n}\right)=n^{3}I_{1}+\OO(n^{2}),\quad
\sum_{(a,b)\in nC}\left(\frac{a}{n}\right)=n^{2}I_{2}+\OO(n)
$$
with $I_{1}$ and $I_{2}$ given in (\ref{int1-2}).
Together with the estimate $d_n=n^{2}\area(C)+\OO(n)$, we get
$$
\log\prod_{(a,b)\in nC}a!=(n^{3}\log n) I_{5}+n^{3}\left(I_{1}-\frac{\log2\pi}{2}\area(C)\right)
+\OO(n^{2}\log n)
$$
where $I_{5}=\int_{C}xdxdy$.
Similarly,
$$
\log\prod_{(a,b)\in nC}b!=(n^{3}\log n) I_{6}+n^{3}\left(I_{1}-\frac{\log2\pi}{2}\area(C)\right)
+\OO(n^{2}\log n)
$$
where $I_{6}=\int_{C}ydxdy$. Moreover,
\begin{align*}
\log \prod_{(a,b)\in nC}  (1+a+b)! &
=\sum_{(a,b)\in nC}\sum_{k=2}^{n+1}
\log\Gamma\left(\frac{a+b+k}{n}\right)
\\
& +\sum_{(a,b)\in nC}\frac{1-n}{2}\log2\pi+\sum_{(a,b)\in nC}\frac{2a+2b+3}{2}\log n.
\\
& =\sum_{(a,b)\in nC}\sum_{k=2}^{n+1}\log\Gamma\left(\frac{a+b+k}{n}\right)
+\left(\frac{1-n}{2}\right)d_n\log2\pi
\\
& +n\log n\sum_{(a,b)\in nC}\left(\frac{a+b}{n}\right)
+\frac32{d_n}{\log n}
\\
& = (n^{3}\log n) I_{4}+n^{3}\left(I_{3}-\frac{\log2\pi}{2}\area(C)\right)+\OO(n^{2}\log n)
\end{align*}
with $I_{3}$ and $I_{4}$ given in (\ref{int3}).
Hence,
$$
\QQ_{n}=n^{3}\left(I_{1}+I_{2}-I_{3}-\frac{\log2\pi}{2}\area(C)\right)+\OO(n^{2}\log n).
$$
Now,
$$\log\delta_{C}(\B)= \lim_{n\to \infty} \frac{\area(C)}{2nd_nM_C}\QQ_{n}$$
 where $d_n=\hbox{dim}\Poly(nC)\simeq n^{2}\area(C)$ and 
 $$M_C=\iint_{C} (x+y)dxdy= I_{4}.$$
Hence
$$
\delta_{C}(\B)=\exp\left(\frac{1}{2I_{4}}\left(I_{1}+I_{2}-I_{3}-\frac{\log2\pi}{2}\area(C)\right)\right),
$$
which is (\ref{val-delt}). 
\end{proof}
 \begin{figure}[!tb]
 \centering
\includegraphics[width=10cm]{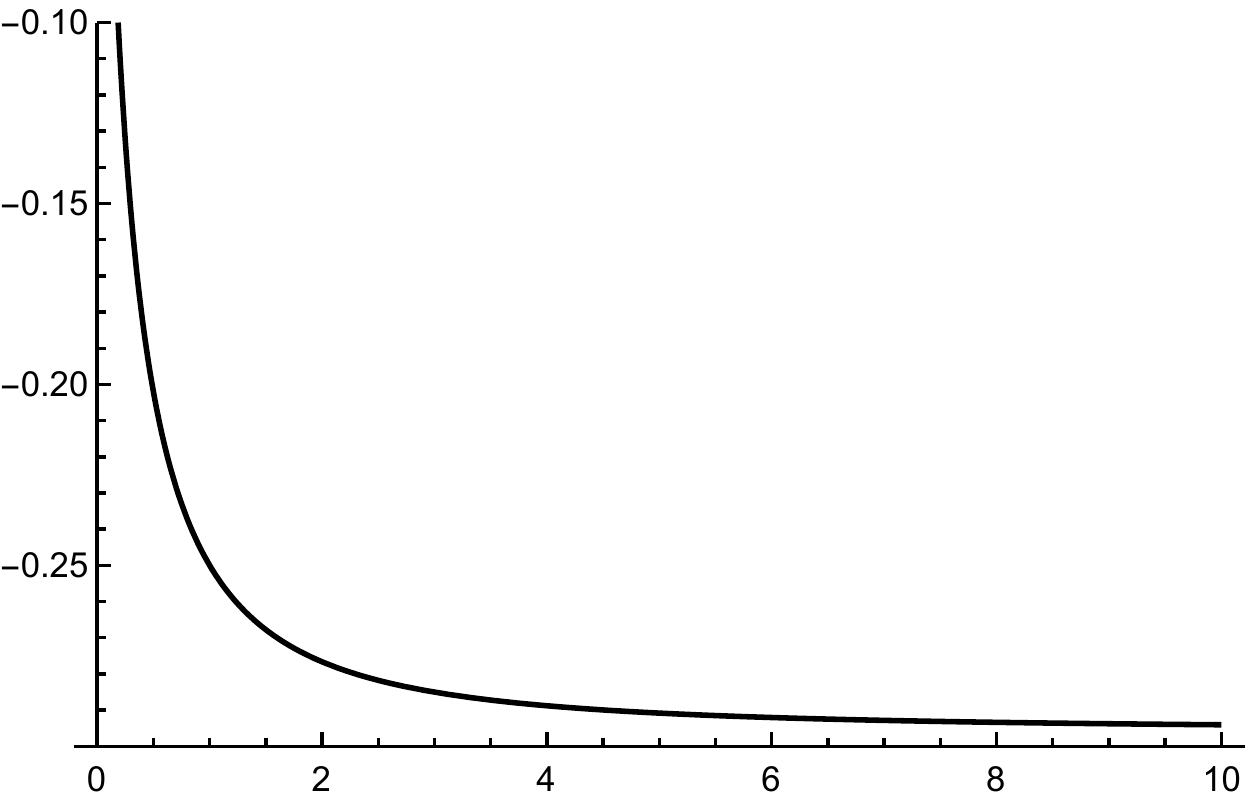}
\caption{$\log\delta_{C_{p}}(\B)$ as a function of $p$. When $p$ goes to $\infty$, $\log\delta_{C_{p}}(\B)$ tends to $(1-4\log2)/6\simeq-0.295$.}
\label{zeros-sym}
\end{figure}

\begin{remark} We have
$$\log\delta_{C_p}(\B)= \lim_{n\to \infty} \frac{\area(C_{p})}{2nd_nM_{p}}\QQ_{n,p}$$
 where $d_n=\hbox{dim}\Poly(nC_p)\simeq n^{2}\area(C_{p})$ and 
 $$M_{C_p}=\iint_{C_p} (x+y)dxdy= 2I_{2}.$$
 Hence
$$
\delta_{C_{p}}(\B)=\exp\left(\frac{3p}{4B(1/p,2/p)}\left(2I_{1}-I_{3}-\frac{\log2\pi}{4p}B(1/p,1/p)\right)\right),
$$
where $B(x,y)$ denotes the Beta function.
\end{remark}

\begin{remark} \label{calcrem} The integrals $I_1,I_2$ and $I_3$ can be simplified, eliminating the Gamma function from the integrand. We illustrate this with $I_1$. To this end, let 
$$F(x):=\int_0^1 \log {\Gamma(x+z)}dz.$$
Then
$$F'(x)= \int_0^1 \frac{\Gamma'(x+z)}{\Gamma(x+z)}dz=\log {\Gamma(x+1)}-\log {\Gamma (x)}=\log x.$$
Thus $F(x)= x(\log {x} -1)+c$ and it follows from the Raabe integral of the Gamma function that $c=\int_{0}^{1}\log\Gamma(z)dz=\frac{1}{2}\log {2\pi}$.
Hence
$$I_1 =\iint_{C} F(x)dx dy= \iint_C \left(x(\log {x} -1)+\frac{1}{2}\log {2\pi}\right)dx dy.$$
In a similar fashion,
$$I_2 =\iint_{C} F(y)dx dy= \iint_C \left(y(\log {y} -1)+\frac{1}{2}\log {2\pi}\right)dx dy \ \hbox{and}$$
$$I_3 =\iint_{C} F(x+y)dx dy= \iint_C \left((x+y)[\log {(x+y)} -1]+\frac{1}{2}\log {2\pi}\right)dx dy.$$
Using these relations and (\ref{target2}), we recover (\ref{ball2}).
\end{remark}

Making use of (\ref{ball2}), we get the following result for the case of $C=C_{p}$, $0\leq p<\infty$.
\begin{proposition} \label{sixsix}
We have
$$
\log\delta_{C_{p}}(\B)=\frac{3p}{2B(1/p,2/p)}\left(\iint_{C_{p}}x\log xdxdy-\iint_{C_{p}}x\log(x+y)dxdy\right),
$$
where $B(x,y)$ denotes the Beta function. In particular, for $p=1, \ p=2$ and $p=\infty$ we get
$$\delta_{C_{1}}(\B)=e^{-1/4},\quad \delta_{C_2}(\B)=\sqrt{2}(\sqrt{2}-1)^{1/\sqrt{2}}, \quad \delta_{C_{\infty}}(\B)=2^{-2/3}e^{1/6}.$$
\end{proposition}
\begin{proof}
One has
$$
\area(C_{p})=\iint_{C_{p}}dxdy=\frac{1}{2p}B(1/p,1/p),\qquad I_{2}=\frac{1}{3p}B(1/p,2/p),
$$
and the given formula follows. The particular values for $p=1,2,\infty$ follow from computing the two integrals for these cases.
\end{proof}

\section{Final remarks} As noted in \cite{LM}, the results given here in sections 2 and 3 on $C-$Robin functions and $C-$transfinite diameter for triangles $C$ in $\R^2$ with vertices $(0,0), (b,0), (0,a)$ where $a,b$ are relatively prime positive integers should generalize to the case of a simplex $C$ which is the convex hull of points $\{(0,...,0),(a_1,0,...,0),...,(0,...,0,a_d)\}$ in $(\R^+)^d$ with $a_1,...,a_d$ pairwise relatively prime (using the appropriate definition of the $C-$Robin function as defined in Remark 4.5 of \cite{LM}). For a product set $K=E_1\times \cdots \times E_d$ in $\C^d$ where $E_j$ are compact sets in $\C$, Proposition 2.4 of \cite{BosLev} gives that 
$$V_{C,K}(z_1,...,z_d)=\max[a_1g_{E_1}(z_1),...,a_dg_{E_d}(z_d)]$$
where $g_{E_j}$ is the Green function for $E_j$. Hence a generalization of Theorem \ref{triprod} will follow. However, unlike the standard ($C=\Sigma$) case, there is no known nor natural way to express a formula for the $C-$extremal function of a product set $K$ when not all of the component sets are planar compacta; e.g., in the simplest such case, $K=E\times F\subset \C^3$ with $E\subset \C^2$ and $F\subset \C$. Nevertheless, it seems that the techniques adopted in sections 5 and 6 using orthogonal polynomials and/or restricting to $d-$circled sets could likely be utilized to find more general product formulas for $C-$transfinite diameters.

\vspace{1cm}
{\obeylines
\texttt{N. Levenberg, nlevenbe@indiana.edu
Indiana University, Bloomington, IN 47405 USA
\medskip
F. Wielonsky, franck.wielonsky@univ-amu.fr
Universit\'e Aix-Marseille, CMI 39 Rue Joliot Curie
F-13453 Marseille Cedex 20, FRANCE }
}

\end{document}